\documentclass[11pt,leqno]{amsart}
\usepackage{amsfonts, amsmath, amsthm, amssymb,xspace}
\usepackage[breaklinks=true]{hyperref}
\usepackage{amsmath,amscd}

\theoremstyle{plain}

\newtheorem{theorem}{Theorem}[section]
\newtheorem{lemma}{Lemma}[section]
\newtheorem{prop}{Proposition}[section]
\newtheorem{cor}{Corollary}[section]
\newtheorem{conjecture}{Conjecture}
\theoremstyle{definition}
\newtheorem{defin}{Definition}[section]

\newtheorem{Assumption}{Assumption}

\newcommand{\bggo}{\mathcal O}

\newcommand{\mf}[1]{\displaystyle{\mathfrak{#1}}}

\DeclareMathAlphabet{\mathscr}{OT1}{pzc}{m}{it}

\DeclareMathOperator{\spec}{\ensuremath{Spec}}
\DeclareMathOperator{\Gr}{\ensuremath{gr}}

\DeclareMathOperator{\ad}{\ensuremath{ad}}

\DeclareMathOperator{\Sym}{\ensuremath{Sym}}

\begin{document}

\title{An analogue of the Kac-Weisfeiler conjecture }
\author{Akaki Tikaradze}
\address{The University of Toledo, Department of Mathematics, Toledo, Ohio, USA}
\email{\tt tikar06@gmail.com}
\begin{abstract} 
In this paper we discuss an analogue of the Kac-Weisfeiler conjecture for a certain
class of almost commutative algebras. In particular, we prove the Kac-Weisfeiler
type statement for rational Cherednik algebras.

\end{abstract}
\maketitle
\section{introduction}

Throughout $\bold{k}=\bar{\bold{k}}$ we be an algebraically closed field of characteristic
$p>2.$ Let $A$ be an affine $\bold{k}$-algebra which if finite over its center $\bold{Z}(A).$
In this setting
one is interested in studying simple modules, in particular their dimensions.
By Schur's lemma all simple $A$-modules are finite dimensional and they are
paramentrized by the corresponding characters of $\bold{Z}(A):$ 
we have a surjective map with finite fibers $\lbrace Irr\rbrace\to \spec \bold{Z}(A),$ from the set of
 isomorphism classes of  simple $A$-modules to the set of characters of $\chi\in\bold{Z}(A).$
For each  character $\chi \in \spec \bold{Z}(A),$ we will denote by $A_{\chi}$
the algebra $A\otimes_{\bold{Z}(A)}\bold{k},$ where $\bold{k}$ is
viewed as a $\bold{Z}(A)$ module via $\chi$. We would like to
study the largest power of $p$ that divides dimensions
of all simple $A_{\chi}$-modules. Let us denote this number by $i(\chi).$
We would like to relate function $i:\spec{\bold{Z}(A)}\to \mathbb{Z}_{+}$ to geometry
of $\spec{\bold{Z}(A)}.$
More specifically, let $\cup S_i=\spec \bold{Z}(A)$ be the smooth
stratification of $\spec\bold{Z}(A)$, 
and for $\chi\in \spec \bold{Z}(A)$ denote by $s(\chi)$ the dimension of the 
smooth stratum containing $\chi.$  

Motivated by the Kac-Weisfeiler conjecture \cite{KW}, which is now a theorem of Premet\cite{P},
it is tempting to state the following.

\begin{conjecture}\label{conjecture}
Supposed that $A$ is a nonnegatively filtered $\bold{k}$-algebra, such that
$\Gr A$ is a finitely generated commutative domain over $\bold{k}.$ Assume that
$(\Gr A)^p\subset \Gr \bold{Z}(A)$ and that $\spec \Gr A$ is a union of
finitely many symplectic leaves and is a Cohen-Macaulay variety, then for any central character
$\chi\in \spec \bold{Z}(A),$ we have $i(\chi)\geq \frac{1}{2}s(\chi).$

\end{conjecture}

Let us explain how does the above relate to the Kac-Weisfeiler conjecture.

Let $\mf{g}$ be a Lie algebra of a semisimple simply-connected algebraic group $G$
over $\bold{k}.$ Then
given that $p$ is large enough, we have Veldkamp's theorem (\cite{V}) describing
the center of the enveloping algebra of $\mf{g}$:
 $\bold{Z}(\mf{u}\mf{g})=\bold{Z}_p(\mf{g})\otimes_{\bold{Z}_p\cap \bold{Z}_{HC}} \bold{Z}_{HC}$
where $\bold{Z}_p$ is the $p$-center, generated by elements $g^p-g^{[p]}, g\in \mf{g},$
and $\bold{Z}_{HC}=\mf{U}\mf{g}^{G}$ is the Harish-Chandra part of the center.
 A character of $\bold{Z}(\mf{U}\mf{g})$ can be thought of as a pair
 $(\chi, \lambda), \chi\in \mf{g}^{*}, \lambda:\bold{Z}_{HC}\to \bold{k}.$
Put $A=\mf{U}\mf{g}/ker(\lambda).$ Then $A$ inherits the filtration from $\mf{U}\mf{g},$
and $\Gr A=\bold{k}[\mathcal{N}],$ where $\mathcal{N}$ is the nilpotent cone of $\mf{g}^{*}.$
Thus, $\Gr A$ is a Cohen-Macaulay domain, and $\spec \Gr A$ is a union
of finitely many symplectic leaves. Thus, assumptions of the conjecture
are met. We have that $\chi\in \spec \bold{Z}(A).$ Now it
follows that $i(\chi)=\dim G\chi.$ Thus  Conjecture \ref{conjecture} in this
case says that any simple $A$-module affording character $\chi$ has dimension 
divisible by $p^{\frac{1}{2}\dim G\chi},$ which is the statement of the Kac-Weisfeiler conjecture.

As a supporting evidence for the above conjecture, we will show
it for $\chi$ in the smooth locus of $\spec \bold{Z}(A)$ (Proposition \ref{props}).
 Also, we will show that for all but finite $\chi\in \spec \bold{Z}(A)$ any
 irreducible representation affording $\chi$ has dimension divisible by $p$ (Corollary \ref{travis}).

Following the approach of Premet and Skryabin \cite{PS}, we prove
the Kac-Weisfeiler type statement (which is weaker than Conjecture \ref{conjecture})
 for a large class of filtered
algebras which includes rational Cherednik algebras (Corollary \ref{cherednik}).

\section{Codimensions of Poisson ideals}

We start by recalling the definition of algebraic symplectic leaves.

\begin{defin} Let $A$ be a Poisson algebra over $\bold{k}.$ Then
a closed symplectic leaf of $\spec A$ is a closed subvariaty
defined by a prime Poisson ideal $I$ such that the Poisson variety $\spec A/I$ is 
a symplectic variety. A closed symplectic leaf of a Poisson
 variety $X$ is a closed subvariety $Z\subset X$ such that
 for any affine open subset $\spec A\subset X, Z\cap \spec A$ is
 a closed symplectic leaf of $\spec A.$ Finally, an algebraic
 symplectic leaf of a Poisson variety $X$ is a closed
 symplectic leaf of an open subvariety of $X.$

\end{defin}

We will recall the definition of Poisson orders by Brown-Gordon \cite{BGo}

\begin{defin} A Poisson order is a pair of an affine $\bold{k}$-algebra $A$ and its central
subalgebra $\bold{Z}_0,$ such that $A$ is finitely generated module over $\bold{Z}_0,$ 
and $\bold{Z}_0$ is a Poisson algebra, $A$ is a Poisson $\bold{Z}_0$-module such that
for $a\in \bold{Z}_0,$ the Poisson bracket $\lbrace{a, \rbrace}$ is a derivation
of $A.$ In this case $A$ is called Poisson $\bold{Z}_0$-order.
\end{defin}
\begin{defin}
Let $A$ be a Poisson $\bold{Z}_0$-order. A Poisson $A$-module is a left $A$-module
M equipped with  a $\bold{k}$-bilinear map $\lbrace, \rbrace:\bold{Z}_0\otimes M\to M$
such that
$$\lbrace \lbrace a, a'\rbrace, m\rbrace=\lbrace a, \lbrace a', m\rbrace\rbrace-\lbrace a', \lbrace a, m\rbrace\rbrace$$
$$\lbrace a, bm\rbrace=\lbrace a, b\rbrace m+b\lbrace a, m\rbrace$$
for all $a, a'\in \bold{Z}_0, b\in A, m\in M.$
\end{defin}

Let $A$ be an associative $\bold{k}$-algebra, and let $\bold{Z}$ be its
central subalgebra. Assume moreover that $A$ is a finitely generated $\bold{Z}$-module.
Let $L\subset Der_{\bold{Z}}(A)$ be a $\bold{Z}$-submodule which
contains all inner derivations and
is closed under the commutator bracket and taking to the $p$-th power.
Thus $L$ is a restricted Lie subalgebra of  $Der_{\bold{Z}}(A).$ Let
$D_{L}(A)$ denote the following algebra. $D'_{L}(A)$ is generated by $L$
as an algebra over $A$ subject to the following relations:
$$i_la-ai_l=l(a), i_{l_1}i_{l_2}-i_{l_2}i_{l_1}=i_{[l_1, l_2]}, $$
$$(i_{l_1})^p=i_{{l_1}^{p}}, a\in A, l_1, l_2\in L, i_{\ad a}-a=0$$
It is immediate that if $L$ is the set of all inner derivations, then
$D'_{L}(A)=A.$

Given a Poisson $\bold{Z}_0$-order we will define algebra 
$D'_{\bold{Z}_0}(A)$ as follows. Let $L\subset Der_{\bold{k}}(A)$ be
the restricted Lie subalgebra of $Der_{\bold{k}}(A)$ generated
by all $\lbrace a, -\rbrace, a\in\bold{Z}_0$ and all inner derivations.
 Then $D'_{\bold{Z}_0}(A)$ denotes $D'_{L}(A).$

We will be primarily interested in two-sided Poisson ideals
of $A,$ i.e. two sided ideals $I\subset A,$ such that $\lbrace z, a\rbrace\in I$ for
all $z\in\bold{Z}, a\in I.$ Clearly given such an ideal $I$, both $I, A/I$ are
$D'_{\bold{Z}_0}(A)$-modules. The following is very standard

\begin{lemma} Let $A$ be a Poisson $\bold{Z}_0$-order. Then $A$ embeds
in $D'_{\bold{Z}_0}(A)$ and $D'_{\bold{Z}_0}(A)$ is finite over
its central subalgebra $(\bold{Z}_0)^p.$

\end{lemma}

\begin{proof}

Let $J$ is the kernel of the map $A\to D'_{\bold{Z}_0}(A).$ Since $A$ is naturally
a $D'_{\bold{Z}_0}(A)$-module, we have that $JA=0,$ so $J=0.$ It is immediate
that $(\bold{Z}_0)^p$ is central in $D'_{\bold{Z}_0}(A),$ and since $Der_{\bold{k}}(A)$
is  a finite $(\bold{Z}_0)^p$-module, we get that $D'_{\bold{Z}_0}(A)$ is finite over
$(\bold{Z}_0)^p.$

\end{proof}

Let $X$ be a variety over $\bold{k}.$ 
Recall the definition of the quasi-coherent sheaf of algebras of the
 crystalline differential operators
$D(X)$. Locally it is generated by  $\bggo{}_X$ and vector fields $TX$ 
subject to the condition
$$[\xi, f]=\xi(f), \xi_1\cdot \xi_2-\xi_2\cdot\xi_1=[\xi_1, \xi_2],$$
where $f, \xi_1$ and $\xi_2$ are local sections of $\bggo{}_X, TX$ respectively.

Given an affine Poisson $\bold{k}$-algebra $S,$ 
We will denote $D'_{S}(S)$ simply by $D'(S).$
 Also, given a Poisson ideal $I\subset S$ we will denote by $D'(S, I)$
  the quotient of
$D'(S)$ by the two sided ideal generated by the image of $I$ in $D'(S).$
Finally, for any closed point $y\in \spec S/I,$ $D'(S, I)_y$ will denote
the quotient of $D'(S, I)$ by the two-sided ideal generated by the 
image of $m_y^p$ (where $m_y$ denotes the maximal ideal corresponding to $y$).

Clearly, $D'(S, I)_y$ is a finite dimensional algebra
over $\bold{k}.$

We will use the following result of Bezrukavnikov-Mirkovic-Rumynin \cite{BMR}

\begin{theorem} [\cite{BMR} Theorem 2.2.3] \label{BMR} Let $X$ be a smooth variety over $\bold{k},$ 
then the ring of crystalline
differential operators $D(X)$ is an Azumaya sheaf of algebras over $T^{*}X^{(1)}$ (the Frobenius
twist of the cotangent bundle of $X$).
\end{theorem}
\begin{cor}
Let $B$ is a finitely generated Poisson domain over $\bold{k},$
such that $\spec B$ is a symplectic variety, then $D'(B)_{y}$ is isomorphic
to the matrix algebra of dimension $p^{2\dim B}$ over $\bold{k}.$
\end{cor}
\begin{proof}

Since $\spec B$ is a symplectic variety, it follows that $D'(B)_{y}$ is isomorphic to the stalk of $D(X)$ 
at $(y, 0)\in T^{*}X^{(1)}.$ Therefore, by Theorem \ref{BMR} we are done. 

\end{proof}
Recall that for any closed point $y$ of a Poisson variety $X$, we
denote by $d(y)$ the dimension of the symplectic leaf of $X$ which contains $y.$
The following will be crucial.

\begin{theorem} \label{key} Let $A$ be a Poisson $S$-order, 
and let $y\in \spec S$ be a closed point. Then any
finite dimensional Poisson $A$-module which as an $S$-module is supported
on $y$ has dimension divisible by $p^{d(y)}.$

\end{theorem}

\begin{proof}

Let us put $d(y)=d.$
For a Poisson algebra $S,$ we will denote by $Der'(S)\subset Der(S)$
the $\bold{k}$-span of all derivations of the form $a\lbrace b, -\rbrace, a, b\in S.$
Let $Y\subset \spec S$ be the symplectic leaf containing $y.$ Let $I$ be a Poisson ideal
corresponding to $\overline{Y}.$ Let $f\in S$ be an element which does not vanish on $y$ and vanishes on
$\overline{Y}-Y.$ Let us put $S'=S_f,$ then $D'(S',I)_y=D'(S, I)_y$ and
$\spec S'$ contains the closed symplectic leaf  through $y.$ Any Poisson $S$-module $M$
supported on $y$ is $D'(S')_y$-module.
Consider a descending filtration $M\supset IM\supset I^2 M\cdots \supset 0.$ 
Since $I\subset m_y,$ we have $I^lM=0$ for some $l.$
Each
quotient $I^jM/I^{j+1}M$ is a module over $D(S', I)_y.$ Therefore, it suffices to
prove that any finite dimensional $D(S', I)_y$-module has dimension
divisible by $p^{d(y)}.$

Denote $S'/I$ by $B.$ Thus, $\spec B$ is a symplectic variety.
We have a natural projection $j:D'(S', I)_{y}\to D'(B)_y.$ 
Since the localization of $B$ at $m_y$ is a regular local ring with the residue field
$\bold{k},$ we have that $B/m_y^p=\bold{k}[x_1,\cdots, x_d]/m^p$, where $m_y=(x_1,\cdots, x_d).$
Remark that $B/(m_y)^p$ is the image of $S'$ in  $D'(S, I)_y.$ 
We will denote the images of $x_1,\cdots, x_d$ in $D'(S, I)$ again by $ x_1,\cdots, x_d.$
Let
$y_1,\cdots ,y_d\in Der'(B/m_y^p)$ be such that $$[y_i, x_j]=\delta_{ij}, [y_i, y_j]=0.$$
Thus, $D'(B)_y$ is generated by $\sum_{a_i<p} {y_1}^{a_1}\cdots {y_d}^{a_d}$ as a
free left (or right) $B/m_y^{p}$-module. Let $\xi_1,\cdots ,\xi_d$ be any lifts
of $y_1,\cdots, y_d$ in the image of $Der'(S')$ in $D'(S, I)_y.$ Notice that
$[\xi_i, x_j]=\delta_{ij}.$ Let us denote by $J$ a $\bold{k}$-subalgebra of $D'(S, I)_{y}$
generated by elements of $Der'(S)$ whose images are in $I.$
Then $J$ is an ideal of the Lie algebra $D'(S, I)^p_{y}$ and $[B/m_y^p, J]=0.$
Denote by $N\subset D'(S, I)^p_{y}$ the $\bold{k}$-span of elements of the 
form ${\xi_1}^{a_1}\cdots\xi_n^{a_d}, a_i<p,$ so $\dim N=p^d.$ 
For any $\xi\in Der'(S),$ its image in $D'(B)_y$ can be written as
$\sum_{i=1}^d b_iy_i$ for some $b_i\in B/{m_y}^p.$ Thus $\xi-\sum b_i{\xi}_i\in J,$
so  $Der'(S)\subset NJ(B/{m_y}^p).$ Therefore, ${D'(S, I)^p}_y=NJ(B/{m_y}^p).$

Let $V$ be a simple $D'(S, I)^p_{y}$-module. Let $v\neq0, v\in V$ be such that
$m_yv=0.$ Then $m_yJv=0,$ so $NJv=V.$ Now we claim that if $v_1,\cdots, v_k\in Jv$
are linearly independent and $\sum_{i=1}^k n_iv_i=0, n_i\in N,$ then $n_i=0$ for all $i.$
Indeed, we have that $\sum_i [x_j, n_i]v_i=\sum_i\frac{\partial n_i}{\partial \xi_j}v_i=0.$
Proceeding by induction on the total degrees of $n_i$ in $\xi_1,\cdots, \xi_n,$ we are done.
Therefore, $\dim V=\dim N\dim(Jv)$ is a multiple of $p^d.$

\end{proof}

\begin{prop}\label{cheeky} Let $M$ be a finite dimensional Poisson module over 
a Poisson $S$-order. Then $\dim M$ is divisible
by inf$\lbrace p^{d(y}, y\in Supp(M)\subset \spec S\rbrace.$

\end{prop}
\begin{proof}
Let us write $M=\oplus_{y\in Supp(M)} M_y,$ where $M_y$ is the submodule of
 elements of $M$ supported on $y.$
Obsereve that $M_y$ is actually a Poisson submodule of $M.$ Indeed,
Let $a\in M_y,$ then $(m_y)^{pi}a=0,$ for some $i.$ Therefore, for any $a\in S,$
$(m_y)^{pi}\lbrace a, m\rbrace=0,$ since $\lbrace a, m_y^{pi}\rbrace=0.$ Now
applying Theorem \ref{key} to each $M_y,$ we are done.
\end{proof}

Recall the following standard definition.

\begin{defin}

A quantization of a Poisson $S$-order $A$ is an $\hbar$-complete  
flat associative
 $\bold{k}[[\hbar]]$-algebra $A',$ such that $A=A'/\hbar A'$ and Poisson bracket of $S$ is induced 
 from the commutator  bracket of $S'.$

\end{defin}

\begin{prop} Let $A'$ be a quantization of a Poisson $S$-order $A.$ 
Let
$J\subset A'[\hbar^{-1}]$ be a two-sided ideal of finite codimension over 
$\bold{k}((\hbar)),$ then \\$\dim_{\bold{k}((\hbar))}$ $A'[\hbar^{-1}]/J$ is a multiple
of inf $(p^{d(y)}, y\in supp A/((A'\cap J)/\hbar).$
\end{prop}

\begin{proof}
Let us put $J'=J\cap A',$ then $A'/J'$ is a free $\bold{k}[[\hbar]]$-module and
$A'[\hbar^{-1}]/J=A'/J'\otimes_{\bold{k}[[\hbar]]}\bold{k}((\hbar)).$ Therefore
$\dim_{\bold{k}((\hbar))} A'[\hbar^{-1}]/J=\dim_{\bold{k}[[\hbar]]} A'/J'.$ Let
us put $J'/\hbar=N\subset A.$ Then $N$ is a Poisson submodule of in $A,$
and $A'/J'$ is a quantization of $A/N.$ Therefore, 
$\dim_{\bold{k}[[\hbar]]} A'/J'=\dim_{\bold{k}}A/N.$
Thus, applying Corollary \ref{cheeky} to $A/N$ we are done.

\end{proof}

\section{Estimates for filtered algebras}

Let $A$ be an associative $\bold{k}$-algebra equipped with
a positive filtration by $\bold{k}$-subspaces 
$A_0\subset A_1\cdots\subset A_n\cdots, A_nA_m\subset A_{n+m}, A=\cup A_n.$
Recall that the center of the associated graded algebra $\Gr A=\oplus A_n/A_{n-1}$
becomes equipped with the natural Poisson bracket and $\Gr A$ is a Poisson
module over it. Indeed, let $d$ be the largest integer such that
$[a, b]\subset A_{n+m-d}$ for any $b\in A_m, a\in A_n, \Gr a\in \bold{Z}(\Gr A).$
Then one puts $\lbrace \Gr a, \Gr b\rbrace=[a, b]/A_{n+m-d-1}\in (\Gr A)_{n+m-d-1}.$

In this section we apply results from the previous section to representations 
of certain class of filtered affine $\bold{k}$-algebras.



Recall  for a filtered algebra $A$ the construction of the Rees algebra 
$R(A)=\oplus_{n}A_n\hbar^n\subset A[\hbar],$
where $\hbar$ is an indeterminate. Then 
$$R(A)/\hbar R(A)=\Gr A, R(A)/(\hbar-\lambda)R(A)=A, R(A)[\hbar^{-1}]=A[\hbar, \hbar^{-1}],$$
for any $\lambda\in \bold{k}, \lambda\neq 0.$
Since $\bold{Z}(A)$ inherits filtration from $A,$ we have $R(\bold{Z}(A))=\bold{Z}(R(A)).$ 
We will fix the embedding $A\subset R(A)[\hbar^{-1}]=A[\hbar, \hbar^{-1}].$
 More
generally, if $Z_0\subset Z(A)$ is a subalgebra such that $\Gr A$ is finite over $\Gr Z_0,$ then
$R(A)$ is finite over $R(\bold{Z}_0).$

We will need the following standard fact.

\begin{prop}\label{Cohen} Let $M$ be a nonegatively filtered module over
a nonnegatively filtered commutative $\bold{k}$-algebra $B,$
such that $\Gr M$ is a finitely generated Cohen-Macaulay module over $\Gr B$ 
and $\Gr B$ is a finitely generated algebra over $\bold{k}.$
Then both $M, R(M)$ are Cohen-Macaulay modules over $\Gr B, R(B)$ respectively.

\end{prop}
We will recall the proof for the convenience of the reader.
\begin{proof}
Let us choose algebraically independent homogeneous elements $x_1,\cdots, x_n\in \Gr B,$
 such that
$\Gr B$ is finite over $\bold{k}[x_1,\cdots, x_n].$ Then $\Gr M$ is a finitely
generated $\bold{k}[x_1,\cdots, x_n]$-module, and since it is a Cohen-Macaulay module
it is a projective (by Lemma 2.2, 2.4 \cite{BBG}), and hence by the Quillen-Suslin theorem a free 
$\bold{k}[x_1,\cdots, x_n]$-module. Let $y_1\cdots, y_m\in \Gr M$ be a
homogeneous basis of $\Gr M$ over $\bold{k}[x_1,\cdots, x_n].$ Let
$a_1,\cdots, a_n, b_1,\cdots b_m$ be the any lifts of 
$x_1,\cdots, x_n, y_1,\cdots, y_m$ in $B, M$ respectively. Then it
follows immediately that $a_1,\cdots, a_n$ are algebraically
independent and $M, R(M)$ is a free $B, R(B)$-module
with basis $b_1,\cdots, b_n$; $b_i\hbar^{\deg{b_i}}$ respectively

\end{proof}

\begin{prop}\label{props} Suppose that $A$ is a nonnegatively
filtered $\bold{k}$-algebra, such that $\Gr A$ an affine commutative Cohen-Macaulay domain.
Suppose also that $\spec \Gr A$ consists of finitely many symplectic leaves. 
If $(\Gr A)^p\subset \Gr \bold{Z}(A)$ then for any character
$\chi$ which belongs to the smooth locus of $\spec \bold{Z}(A),$ a simple
A-module affording $\chi$ has dimension $p^{\frac{1}{2}\dim A}$.

\end{prop} 

\begin{proof} 
Applying [\cite{T} Lemma 2.4], we get that $\Gr \bold{Z}(A)=(\Gr A)^p.$
By [\cite{T} Theorem 2.3], the complement of the Azumaya locus of $\bold{Z}(A)$
has codimension 2, and the largest possible dimension of $A$-module is 
$p^{\frac{1}{2}\dim \Gr A}$.
Let $U\subset \spec\bold{Z}(A)$ be the smooth locus of $\spec\bold{Z}(A).$
We claim that $A_{|U}$ is a locally free sheaf over $U.$ Indeed,
by the assumption $\Gr A$ is a Cohen-Macaulay module over $(\Gr A)^p.$
But then by \ref{Cohen} $A$ is a Cohen-Macaulay module over $\bold{Z}(A),$ 
since $\Gr\bold{Z(A)}=(\Gr A)^p.$
So $A|_{U}$ is a Cohen-Macualay over $U,$ and since $U$ is nonsingular, we get
that $A|_{U}$ is locally free over $U.$ To summarize, $A|_{U}$ is
locally free and its Azumaya locus has complement of codimension
$\geq 2.$ Therefore, $A|_{U}$ is an Azumaya algebra over $U$ by [\cite {BG} Lemma 3.6].

\end{proof}

From now on in this section we will follow very closely Premet and Skryabin
 \cite{PS}.
We will use the following result from \cite{PS}.

\begin{lemma}[\cite{PS} Lemma 2.3]\label{premet} Let $A$ be finite and projective over its central
affine subalgebra
$Z_0,$ and let $L\subset Der_{\bold{Z_0}}(A)$ be a restricted Lie subalgebra of derivations. Then for any 
$i\geq 0,$ the set of characters $\chi\in \spec \bold{Z}_0,$ such that $A_{\chi}$
has an $L$-stable two sided ideal of codimension $i$ is closed.

\end{lemma}
 Of course, $L$-stable two-sided ideal of $A$ is the same as a $D'_L(A)$-submodule
 of $A.$

From now on in this section, by
 $L\subset Der_{R(\bold{Z_0})}(R(A))$ we will always denote the restricted Lie algebra 
generated by all inner derivations and by ${\hbar}^{-d}\ad(a),$ for all
$a\in R(A),$ such that $a/\hbar\in \bold{Z}(\Gr A).$ In this setting,
we have the following

\begin{lemma} $D'_{L}(R(A))/\hbar=D'(\Gr A)$ and $D'_{L}(R(A))[\hbar^{-1}]=A[\hbar, \hbar^{-1}].$

\end{lemma}
\begin{proof}
First equality is immediate. The second equality follows from
the fact that $L$ consists of inner derivations in $Der(A[\hbar, \hbar^{-1}]).$

\end{proof}

We will use the following

\begin{Assumption}\label{H}

Let $A$ be a positively filtered $\bold{k}$-algebra, such
that $\bold{Z}(\Gr A)$ is finitely generated over $\bold{k}$ and $\Gr A$ is
a finitely generated module over $\bold{Z}(\Gr A).$ Moreover, assume that
there in a positive integer $n,$ such that $(\bold{Z}(\Gr A))^{pn}\subset \Gr \bold{Z}(A).$

\end{Assumption}

In what follows, we will use the following standard notations. For a subset $W$ of an affine variety $X,$ 
we will denote by $I(W)$ the reduced ideal
of zeros of $W$ in $\bggo{}(X),$ and for an ideal $I\subset \bggo{}(X), V(I)$ will denote the
set of zeros of $I$ in $X.$

We have the following

\begin{theorem}\label{sheep} Let algebra $A$ satisfy
the assumption \ref{H}. Let $\bold{Z}_0\subset \bold{Z}(A)$ be a subalgebra such that
$R(A)$ is finitely generated projective module over $R(\bold{Z}_0).$ 
Let $W\subset \spec \bold{Z}_0$ be a closed subset consisting of
points $\chi\in W$ such that $A_{\chi}$ has a two sided ideal of codimension $p^ij$ 
where $p$ does not divide $j,$ then for any 
$y'\in V(\Gr I(W))\subset \spec \Gr \bold{Z_0},$ there is $y\in \spec \bold{Z}(\Gr A)$ 
which is in the preimage of
$y'$ in $\spec \bold{Z}(\Gr A)\to \spec \Gr \bold{Z}_0,$ such that $d(y)\leq i.$
\end{theorem}

\begin{proof}
Recall that $R(z_0)[{\hbar}^{-1}]=\bold{Z}_0[\hbar, \hbar^{-1}].$
We again denote by $L$ the Lie subalgebra of Der(R(A)) as in above lemma.
$\spec \bold{Z}_0\ltimes (\mathbb{A}^1-{0})\hookrightarrow \spec R(\bold{Z}_0).$ 
We a hve a map $\rho: \bold{Z}(\Gr A)\to HH^1(R(A))$ defined
as $\rho(a)=\frac{1}{{\hbar}^d}\ad(a'),$ where $a'\in p^{-1}(a), a\in \bold{Z}(\Gr A).$
and $p:R(A)\to R(A)/\hbar=\Gr A$ is the quotient map.
Let $X$ denote the set of all closed points $\chi\in \spec R(\bold{Z_0})$ such that
$R(A)_{\chi}$ has a two sided $L$-stable ideal of codimension $p^ij.$
By lemma \ref{premet}, $X$ is closed.
We claim that $X\cap \lbrace \hbar=0 \rbrace$ is the set of all characters 
$\chi:\Gr\bold{Z_0}\to k$ such that there is a Poisson ideal in $\Gr A$
containing $ker(\chi)$ of codimension $p^ij.$ Indeed, for two-sided ideal
$I\subset\ Gr$ to be closed under the map $\lbrace{a, \rbrace}$ for all
$a\in \bold{Z}(\Gr A)$ is the same as ideal $p^{-1}(I)$ being $L$-stable.
Similarly, $X\cap \lbrace \hbar\neq 0$ consists of characters 
$\chi:\bold{Z}_0(A)[\hbar, \hbar^{-1}]\to \bold{k}$ such that
$A_{\chi}$ has a two-sided ideal of codimension $p^ij.$ Indeed,
given a two sided ideal $I\subset A$ which contains $ker(\chi)\cap \bold{Z}_0(A)$
ideal $p^{-1}(I)$ is $L$-stable, since $\hbar$ is invertible
on $R(A)/p^{-1}(I),$ where $p:R(A)\to R(A)/(\hbar-\chi(\hbar))=A$ is the
quotient map.

Therefore,
$W\times (\mathbb{A}^1-{0})\subset X,$  In particular, 
$\overline{W\times (\mathbb{A}^1-{0})}\cap (\hbar=0)\subset X.$
 But $\overline{W\times (\mathbb{A}^1-{0})}\cap (\hbar=0)$
is precisely $V(\Gr I(W)).$ Thus for any character $\chi \in V(\Gr(I(W)))$
Algebra $(\Gr A)/{m_{\chi}}(\Gr A)$ has a Poisson ideal of codimension
$p^ij,$ where $m_{\chi}$ is the maximal ideal of $\Gr \bold{Z}_0\subset (\Gr A)^p$
corresponding to $\chi.$. Let us choose such an ideal $J\subset (\Gr A)/{m_{\chi}}^p.$ 
Let $\lbrace y_i, i\in I\rbrace$ be the (finite)  preimage of $\chi$ under the map 
$\spec \Gr A\to \spec \Gr \bold{Z}_0.$ Then $\Gr A/J$ is a Poisson
$\bold{Z}(\Gr A)$- module
supported on $\lbrace y_i \rbrace.$ 
Applying Corollary \ref{cheeky}, we are done,

\end{proof}
Recall that for a character $\chi,$ $i(\chi)$ denotes the largest power of $p$
which divides all dimensions of irreducible representations affording $\chi.$

\begin{cor}\label{baboon} 
Let $A$ satisfy the assumption \ref{H}.
Let $\bold{Z}_0$ be a subalgebra of $\bold{Z}(A)$ such that $R(A)$
 is a finitely generated projective module over $R(\bold{Z}_0).$ 
 Let $G\subset Aut(R(A), R(\bold{Z}_0))$ be a subgroup of the group of 
 $\bold{k}[\hbar]$-algebra automorphisms of $R(A)$ which preserve $R(\bold{Z}_0).$
Then, for any $\chi\in \spec \bold{Z}_0,$ and any 
$y\in V(\Gr I(\overline{G\chi}))\subset \Gr \bold{Z}_0,$ we have $i(\chi)\geq inf d(y'),$ where
 $y'\in \spec \bold{Z}(\Gr A)$ runs through preimages of $y$
under the map $\spec \bold{Z}(\Gr A)\to \spec \Gr \bold{Z}_0.$
\end{cor}

\begin{proof} Let $W\subset \spec \bold{Z}_0$ be a set of characters $\chi'$
 such that $i(\chi')=i(\chi).$ So $\chi\in W$ and $W$ is closed by Lemma \ref{premet}.
Therefore $\overline{G\chi}\subset W,$ and by Corollary \ref{cheeky}, we are done.

\end{proof}

\begin{cor}\label{travis}
Let $A$ be satisfy the assumption \ref{H}. 
 Then if $\spec \bold{Z}(\Gr A)$ is a union of algebraic 
symplectic leaves, and if $\spec \bold{Z}(\Gr A)$ has at most one (the origin) zero dimensional
symplectic leaf, then all but finitely many irreducible representations
of $A$ have dimension divisible by $p.$

\end{cor}

\begin{proof}

Recall that given a finite dimensional $\bold{k}$-alebra $S,$ 
all simple $S$-modules have dimensions divisible by $p$  if and only if 
$1\in [S, S].$ Therefore, we are
required to show that the support of $1\in A/[A, A]$ on $\spec \bold{Z}(A)$ is
a finite set. Let $Y$ be the support of $1\in D'_{L}(R(A)/[D'_{L}(R(A), D'_{L}(R(A)]$ on
$\spec \bold{Z}(D'_{L}(R(A)).$ We have that $Y\cap \lbrace \hbar\neq 0\rbrace$ is 
$X\times({\mathbb{A}-\lbrace 0\rbrace}),$ where $X$ is the support of $1\in A/[A, A]$
 on $\spec \bold{Z}(A).$
On the other hand,  $Y\cap \lbrace \hbar=0\rbrace$ is the support of 
$1\in D'(\Gr A)/[D'(\Gr A), D'(\Gr A)]$ in $\spec \bold{Z}(D'(\Gr A)).$
By the following lemma, the letter set is finite, so $X$ has to be finite and we are done.

\begin{lemma} Under the assumption of the theorem, 
$1\in D'(\Gr A)/[D'(\Gr A), D'(\Gr A)]$ is supported on the
origin of $\spec \bold{Z}(\Gr A)^p.$

\end{lemma}
\begin{proof} 
 Indeed, if
$y\in \spec \bold{Z}(\Gr A)^p, y\neq 0$ (where 0 denotes
the origin of $\spec \bold{Z}(\Gr A)$), then all representations
of $D'(\Gr A))/{m_y}^p$ are divisible by $p^{d(y)}$(as in the
proof of Theorem \ref{key}. But since $d(y)>0,$
we conclude that all simple $D'(\Gr A))/{m_y}^p$-modules have dimensions divisible by $p.$
Therefore,\\ 
$1\in D'(\Gr A))/[D'(\Gr A), D'(\Gr A)]+{m_y}^pD'(\Gr A)$
so $y\notin Supp(1).$

\end{proof}
\end{proof}

In particular, the assumptions of the above are satisfied when $\spec \bold{Z}(\Gr A)$ consists
of finitely many symplectic leaves. This result can be thought of as a positive 
characteristic analogue of a result due to Etingof-Schedler \cite{ES}, which states
that if $A$ is a nonnegatively filtered algebra over $\mathbb{C}$ such that
$\Gr A$ is finite over its center and $\spec \bold{Z}(\Gr A)$
 is a union of finitely many symplectic leaves, then $A$ has a finitely
many nonisomorphic finite dimensional simple modules.

\section{Applications to rational Cherednik algebras}

At first, we will recall the decomposition of $V/W$ into symplectic
leaves [\cite{BGo}, Proposition 7.4], where $V$ is a symplectic vector space over $\bold{k}$
and $W\subset Sp(V)$ is a finite group such that $p$ does not divide $|W|$.
 Given a subgroup $H\subset  W$
denote by $V^H_{0}$ the set of all $v\in V^{H}$ such that the stabilizer of $v$ is
$H.$ Then we have the decomposition of $V/W$ into symplectic leaves
 $V/W=\cup V^H_{0}/W,$ where $H$ ranges through all conjugacy
classes of subgroups of $W.$ In particular, for $v\in V/W$ dimension
of the symplectic leaf through $v$ is $\dim V^{W_v},$ where $W_v$ is
(a conjugacy class) the stabilizer of $v.$

Recall that for a class function $c:W\to \bold{k},$ Etingof-Ginzburg 
defined a symplectic reflection algebra $H_c(W, V).$ This
algebra comes equipped with the natural filtration such that
$\Gr H_c(W, V)=\bold{k}[W]\ltimes \Sym V,$ (the PBW property, \cite{EG} Theorem 1.3).

\begin{prop} Let $W\subset Sp(V)$ be a finite group.
If $p$ does not divide $|W|,$ then for 
a symplectic reflection algebra $H_{c}(W, V)$, 
there are only finitely many irreducible representations
whose dimension is not divisible by $p.$

\end{prop}

\begin{proof}

The center of $\Gr H_{1, c}=\bold{k}[W]\ltimes \Sym V$ is
$(\Sym V)^{\Gamma},$ so $\Gr H_c(W, V)$ is finite over its center.
 Also, by a theorem of Etingof (\cite{BGF}, Theorem 9.1.1)
  $\Gr \bold{Z}(H_{1, c})=((\Sym V)^{\Gamma})^p.$ Since as explained
  above, $\spec (\Sym V)^{\Gamma}=V/W$ is a union of finitely many
  symplectic leaves, all the assumptions of \ref{travis} are satisfied and we are done.

\end{proof}

Important class of symplectic reflection algebras consists of Rational Cherednik algebras.
Let us recall their definition.

Let $\mf{h}$ be a finite dimensional vector space over $\bold{k}.$ 
Given a finite group $W\subset GL(\mf{h})$, let $S\subset \Gamma$ be 
the set of pseudo-reflections (recall that $s\in W$ is a pseudo-reflection if $Im(Id-s)$ 1-dimensional). 
Let $\alpha_s\in \mf{h}^{*}$ be the generator of $Im(s-1)|_{\mf{h}^{*}}$ and 
$\alpha_s' \in \mf{h}$ be the generator of $Im(Id-s)$ such that $(\alpha_s, \alpha_s')=2.$ Let
$c:S\to \bold{k}$ be a $W$-invariant function.
The rational Cherednik
algebra, $H_{c}(W, \mf{h})$, as introduced by Etingof and Ginzburg \cite{EG}, 
is the quotient of the skew group
algebra of the tensor algebra, $\bold{k}[W]\ltimes T (\mf{h} \oplus \mf{h}^{*})$, 
by the ideal generated by the relations
$$[x, x']=0, [y, y']=0, [y, x]=(y, x)-\sum c_s (y, \alpha_s)(x, \alpha_s')s,$$
$x, x'\in \mf{h}, y, y'\in \mf{h}^{*}.$

There is a standard filtration on $H_c$ given by setting $\deg x=1, \deg y=1, deg (g)=0$ for
all $x\in V, y\in \mf{h}^{*}, g\in \Gamma.$ The PBW property of $H_c(W, \mf{h})$ says that 
$\Gr H_c=\bold{k}\Gamma\ltimes \Sym (\mf{h}\oplus \mf{h}^{*})$. When $c$ is identically
0 then $H_0(W, \mf{h})=\bold{k}\ltimes D(\mf{h}),$ where $D(\mf{h})$ is the ring of
(crystalline) differential operators on $\mf{h}.$

Algebra $H_c(W, \mf{h})$ has a distinguished central subalgebra $\bold{Z}_0=\bold{Z'}\otimes\bold{Z"},$
where $\bold{Z}^{'}$(respectively $\bold{Z}"$) denotes $((\Sym \mf{h}^{*})^W)^p(((\Sym \mf{h})^W)^p$
[\cite{BC}, Proposition 4.2].

For a rational Cherednik algebra $H_{c}(W, \mf{h})$ 
 we have
the spherical subalgebra $B_{c}=e H_{c}(W, \mf{h})e$ $(e=\frac{1}{|W|}\sum g\in W)$. 
Algebra $B_{c}$ inherits a filtration from $H_{c}(W,\mf{h})$ and 
$\Gr B_{c}=(\Sym (\mf{h}\oplus \mf{h}^{*}))^{W}.$
Then $e\bold{Z_0}\subset \bold{Z}(B_c)$

\begin{cor}\label{cherednik} Assume that $p$ does not divide $|W|.$ 
Let $\chi:e\bold{Z}^{'}\to \bold{k}$ be a central character. Let us identify 
$\chi:e((\Sym \mf{h}^{*})^{W})^p\to \bold{k}$
with a point in $\mf{h}/W.$ Let $W_{\chi}$ be the stabilizer in $W$
of a preimage of $\chi$ in $\mf{h}.$ Then all irreducible representations of $B_c$ affording $\chi$
have dimensions divisible by  $p^{\mf{h}^{W_\chi}}.$
\end{cor}

\begin{proof}

We need to prove that given a character $\mu:e\bold{Z}_0=e\bold{Z'}\otimes e\bold{Z"}\to \bold{k},$
such that $\mu|:e\bold{Z'}=\chi,$ then any $B_c$-module affording $\mu$ has
dimension divisible by $p^{{\mf{h}}^{W_{\chi}}}.$ Let us write $\mu=(\chi, \chi'),$
$\chi'\in \spec e\bold{Z"}.$

Since, $\Sym (\mf{h}\oplus \mf{h}^{*})^{W}=\Gr B_{ c}$ is a Cohen-Macaulay algebra, it
is a Cohen-Macaulay module over $\Gr e\bold{Z}_0,$ hence by Proposition
\ref{Cohen} $R(B_{1, c})$ is a Cohen-Macauly module over $R(e\bold{Z}_0).$ However,
$R(e\bold{Z}_0)$ is a polynomial algebra, therefore $R(B_{1,c})$ is a projective (actually
free by the Quillen-Suslin theorem) $R(e\bold{Z}_0)$-module.
$e\bold{Z}_0$ is a polynomial algebra, it follows that $B_c$ is projective (actually free)
over $e\bold{Z}_0.$
 We have an action of $\mathbb{G}_m$ on $H_c(W, \mf{h})$ which preserves $B_{c}$
 corresponding
to the grading with $\deg(x)=1, \deg(y)=-1, \deg(g)=0, g\in \Gamma=0, x\in \mf{h}, y\in \mf{h}^{*}.$
 This action preserves $e\bold{Z}_0.$ Therefore we may apply \ref{baboon}.
We need to understand $V(\Gr I(\mathbb{G}_m\mu))\subset \spec e\Gr Z_0$ and its preimage in $\spec \Gr B_c.$

We have $I(\mathbb{G}_m\mu)\cap e\bold{Z'}=I(\mathbb{G}_m\chi).$ Then 
$\Gr I(\mathbb{G}_m\mu)\cap e\bold{Z'}=\Gr I(\mf{G}_m\chi).$
Indeed, clearly $\Gr I(\mathbb{G}_m\chi)\subset \Gr I(\mathbb{G}_m\mu)\cap \Gr e\bold{Z'}.$ Suppose
that $f\in \Gr I(\mathbb{G}_m\mu)\cap e\bold{Z'}.$ Therefore, there is $g\in \Gr e\bold{Z"}$ such that
$\deg g<\deg f$ and $f+g\in I(\mathbb{G}_m\mu).$ Thus, $f(t\chi)+g(t^{-1}\chi',t\chi)=0, t\in \bold{k}^{*}.$
But the letter is a Laurent polynomial with leading term $t^{\deg f}f(\chi).$ Hence
$f(\chi)=0,$ so $f\in \Gr I(\mathbb{G}_m\chi).$
Therefore $p(V(\Gr I(\mathbb{G}_m\mu))=V (\Gr I(\mathbb{G}_m\chi))=\bold{k}\chi,$
where $p:\spec (e\bold{Z'}\otimes \bold{Z"})\to \spec e\bold{Z'}$ is
the projection. Thus we conclude that there is $\chi"\in \spec e\bold{Z"}$, 
such that
Let $(\chi", \chi)\in V(\Gr I(\mathbb{G}_m\mu).$ Then if $v\in (\mf{h}\times \mf{h^{*}})^{W}$  is a
 preimage of $(\chi", \chi)$
under the map $\spec \Gr B_c=(\mf{h}\times \mf{h^{*}})^{W}\to \spec (e\bold{Z'}\otimes e\bold{Z"}),$
 then the projection of $v$ its projection on $\mf{h}/W$
is $\chi.$

We conclude that $W_{v}$ is a subgroup of $W_{\chi},$ so 
$(\mf{h}\times\mf{h}^{*})^{W_{\chi}}\subset(\mf{h}\times\mf{h^{*}})^{W_{v}}.$
Therefore, Using the description of symplectic leaves of $(\mf{h}\oplus \mf{h}^{*})/W$ 
discussed above, we conclude  $d(v)\geq 2\dim \mf{h}^{W_{\chi}}.$
So, applying \ref{baboon}
$H_{c}(W, \mf{h})$-module which affords
character $\chi$ has dimension divisible by $p^{\frac{1}{2}}d(v),$ therefore
it is divisible by $p^{\dim \mf{h}^{W_{\chi}}}$.

\end{proof}

We have a similar result for Cherednik algebras
\begin{cor}\label{CH} Assume that $p$ does not divide $|W|.$ 
 Let $\chi :\bold{Z'}\to \bold{k}$ be a character.
Then any irreducible representation of $H_{c}(W, \mf{h})$ 
affording $\chi$ has dimension divisible by
$|W/W_{\chi}|p^{\dim \mf{h}^{W_{\chi}}},$ where $W_{\chi}$ is a subgroup fixing
an element of $W$-orbit corresponding to $\chi$ viewed as an element of $\mf{h}/W.$

\end{cor}
\begin{proof}
Sinnce $\Gr H_{c}(W, \mf{h})$ is a free $\Sym (\mf{h}\oplus\mf{h^*})$-module, $\Gr H_{c}(W, \mf{h})$ is
a Cohen-Macauly $\Gr (\bold{Z}_0)$-module, hence $R(H_{c}(W, \mf{h})$ is a projective
$R(\bold{Z}_0)$-module.
Just like in the proof of \ref{cherednik}, using \ref{baboon} we obtain that
 any simple ${H_{c}(W, \mf{h}}_{\chi}$-module $M$ has dimension divisible by $p^{\dim \mf{h}^{\Gamma_v}}.$
But as $\bold{k}[\mf{h}]$-module, $M$ can be written as $M=\oplus_{\chi'\in p^{-1}(\chi)} M_{\chi'},$ where 
$p:\mf{h}\to \mf{h}/W$ is the projection. 
Clearly, if $m\in M_{\chi'},$ then $gm\in M_{g\chi'}.$ So action by elements of $W$ is permuting $M_{\chi'}.$
Hence $\dim M_{\chi'}=\dim M_{g\chi'}, g\in W.$
Since $|p^{-1}(\chi)|=|W/W_{\chi}|,$  
 we conclude that
$|W/W_{\chi}|$ divides $\dim_{\bold{k}} M.$ So, $|W/W_{\chi}|p^{\mf{h}^{W_{\chi}}}$ divides $\dim M.$
\end{proof}

As pointed out to us by I. Gordon, one can also prove Corollary \ref{CH} as follows. 

As before, let $\chi$ be a character of $\bold{Z'}=((\Sym \mf{h}^{*})^W)^p.$ 

We
will denote by $\bold{Z}_c$ the center $H_c(W, \mf{h}).$
Denote by $\overline{H_c(W, \mf{h})_{\chi}}$ the
completion of $H_c(W, \mf{h})$ with respect to $Ker(\chi)\subset ((\Sym \mf{h}^{*})^W)^p.$
Denote by $\overline{\bold{Z}_{c, \chi}}$ the center of $\overline{H_c(W, \mf{h})_{\chi}}.$ 
Then $\overline{\bold{Z}_{c, \chi}}$
is the completion of $\bold{Z}_c$ with respect to the ideal $Ker(\chi)\subset\bold{Z}_c.$ 
We have the similar notations
for $\chi\in \spec((\Sym \mf{h})^W)^p.$ 
By $0\in \spec \bold{Z'}(\spec \bold{Z"})$ we will denote the origin.
The following result is the characteristic $p$ version of
a result by Bezrukavnikov-Etingof \cite{BE}, whose prove is identical to the original one.

\begin{theorem}\label{BE}(\cite{BE} Theorem 3.2) For any $\chi\in \spec((\Sym \mf{h}^{*})^W)^p$ (respectively $\spec((\Sym \mf{h})^W)^p.$
there is an isomorphism of algebras 
$\overline{H_c(W, \mf{h})_{\chi}}\simeq Mat_{|W|/|W_{\chi}|}(\overline{H_{c'}(W_{\chi}, \mf{h})_0}),$
where $W_{\chi}$ is the stabilizer of a lift of $\chi$ in $\mf{h}$(respectively $\mf{h^{*}}$) and $c'$ is the
restriction of $c$ on $W_{\chi}.$

\end{theorem}

In particular, since
 $H_{c'}(W_{\chi}, \mf{h})\simeq Mat_{|W/W_{\chi}|}(D(\mf{h}^{W_{\chi}})\otimes H_{c'}(W_{\chi}, ({\mf{h}}^{W_{\chi}})^{\perp})$
where $({\mf{h}}^{W_{\chi}})^{\perp}=(({\mf{h}}^{*})^{W_{\chi}})^{*}$ and 
and dimension of any finite dimensional $Mat_{|W/W_{\chi}|}(D(\mf{h}^{W_{\chi}})$-module is a multiple
of $|W/W_{\chi}|p^{\dim \mf{h}^{W_{\chi}}},$ we get that \ref{BE} implies \ref{CH}.

We have the following

\begin{cor} For any character $\mu\in \spec \bold{Z}_c,$ there is a subgroup $W'\subset W$ and a character
$\mu'\in \spec \bold{Z}_{c'},$ $\mu'|(\bold{Z'}\otimes \bold{Z"})_{+}=0$ where $c'$ is 
the restriction of $c$ on $W',$ such that 
$\overline{H_c(W, \mf{h})}_{\mu}\simeq Mat_{|W|/|W'|}(\overline{H_{c'}(W', \mf{h})}_{\mu'}.$

\end{cor}
\begin{proof}
Let $\chi$ be the restriction of $\mu$ on $\bold{Z'}.$ By \ref{BE}, we have
$\overline{H_c(W, \mf{h})_{\chi}}\simeq Mat_{|W|/|W_{\chi}}(\overline{H_{c'}(W_{\chi}, \mf{h})_0}).$
Let $\mu_1$ be the character of $\overline{\bold{Z}_{c',0}}$ corresponding to $\mu$ under
the isomorphism $\overline{\bold{Z}_{c, \chi}}\simeq \overline{\bold{Z}_{c', 0}}.$ We will also denote
by $\mu_1$ the restriction of $\mu_1$ on $\bold{Z}_{c'},$ in particular 
$\mu_1{{(\Sym \mf{h^{*}}^{W_{\chi}})}^p}_{+}=0.$
Thus, $\overline{H_c(W, \mf{h})_{\mu}}\simeq Mat_{|W|/|W_{\chi}|}(\overline{H_{c'}(W_{\chi},\mf{h})})_{0}.$ 
Therefore, by replacing $\mu$ with $\mu_1$ we may assume that $\chi=0.$ Again using \ref{BE}
we have $\overline{H_c(W, \mf{h})_{\chi'}}\simeq Mat_{|W|/|W_{\chi'}|}(\overline{H_{c'}(W_{\chi'}, \mf{h})_0})$
Let $\mu'$ be the corresponding character of $\bold{Z}_{c'}.$
 Then $\mu'|(\bold{Z'}\otimes\bold{Z"})_{+}=0$ and we are done.

\end{proof}
The above corollary implies that for studying irreducible representation of Cherednik
algebras it is enough to consider restricted representations, i.e modules
annihilated by ${({\Sym\mf{h}}^W\otimes {\Sym\mf{h^{*}}}^W)^p}_{+}.$

\end{document}